\documentclass[11pt,letterpaper]{article}

\pdfoutput=1
\usepackage[margin=2cm]{geometry}
\date{}
\usepackage{amsthm}
\usepackage{graphicx} % for pdf, bitmapped graphics files
\usepackage{tikz,pmat}
\usepackage{amsmath}
\usepackage{amssymb}
\usepackage{array}
\usepackage{cite}
\usepackage{optidef}

\usepackage{algorithm}
\usepackage{algorithmic}
\usepackage{mathtools}

\usepackage{todonotes}
\newcommand{\mynioptions}{[color=blue!20,inline]}
\newcommand{\mynoptions}{[color=blue!20]}
\newcommand{\myni}{\expandafter\todo\mynioptions}
\newcommand{\myn}{\expandafter\todo\mynoptions}

\usepackage{xspace}
\usepackage{amsmath,amssymb,amsfonts}

\newcommand{\bs}[1]{\ensuremath{\boldsymbol{#1}}}

\newcommand{\pci}[1]{\ensuremath{\bs{p}_{\boldsymbol{c}_{i}}}\xspace}

\usepackage[utf8]{inputenc}
\usepackage[english]{babel}
\usepackage{multirow}

\definecolor{ao}{rgb}{0.0, 0.5, 0.0}
\newtheorem{theorem}{Theorem}
\newtheorem{remark}{Remark}
\newtheorem{lemma}{Lemma}

\newtheorem{problem}{Problem}
\newtheorem{question}{Question}

\newcommand{\R}{\mathbb{R}}
\newcommand{\N}{\mathbb{N}}
\newcommand{\U}{\mathcal{U}}
\newcommand{\W}{\mathcal{W}}
\newcommand{\X}{\mathcal{X}}
\newcommand{\mcS}{\mathcal{S}}
\newcommand{\mcT}{\mathcal{T}}
\newcommand{\mcR}{\mathcal{R}}
\newcommand{\mcP}{\mathcal{P}}
\newcommand{\mcV}{\mathcal{V}}
\newcommand{\mcC}{\mathcal{C}}

\newcommand{\Prob}{\mathbb{P}}
\newcommand{\Exp}{\mathbb{E}}

\newcommand{\Costpi}{r_{x_0}^{U}(\mcS, \mcT)}

\everypar{\setbox0\lastbox}

\title{Voronoi Partition-based Scenario Reduction for Fast Sampling-based Stochastic Reachability Computation of LTI Systems}
\author{Hossein Sartipizadeh, Abraham P. Vinod,  Beh{\c{c}}et A{\c{c}}{\i}kme{\c{s}}e, and Meeko Oishi
	\thanks{This material is based upon work supported by the National Science Foundation, the Air Force Office of Scientific Research, and the Office of Naval Research.  Hossein Sartipizadeh and Beh{\c{c}}et A{\c{c}}{\i}kme{\c{s}}e were supported by Air Force Research Laboratory grant FA8650-15-C-2546 and the Office of Naval Research (ONR) Grant No. N00014-15-IP-00052. %\newline
    Vinod and Oishi were supported under NSF Grant Number CMMI-1254990, NSF Grant No. IIS-1528047, and AFRL Grant No. FA9453-17-C-0087. Any opinions, findings, and conclusions or recommendations expressed in this material are those of the authors and do not necessarily reflect the views of the National Science Foundation. \newline
    H. Sartipizadeh (corresponding author) is with University of Texas at Austin, TX, US. Email: {\tt hsartipi@utexas.edu}.\newline 
		A. Vinod and M. Oishi are with the Electrical \& Computer Engineering, University of New Mexico, Albuquerque, NM, US. Email: {\tt aby.vinod@gmail.com; oishi@unm.edu}.\newline
		B. A{\c{c}}{\i}kme{\c{s}}e is with the Department of Aeronautics \& Astronautics in the University of Washington, Seattle, WA. Email: {\tt behcet@uw.edu.}\newline
	}
}
\begin{document}
	
	\maketitle
	\thispagestyle{empty}
	\pagestyle{empty}
	%%%%%%%%%%%%%%%%%%%%%%%%%%%%%%%%%%%%%%%%%%%%%%%%%%%%% 
	%%%%%%%%%%%%%%%%%%%%%%%      Abstract    %%%%%%%%%%%%%%%%%%%%%%%
	%%%%%%%%%%%%%%%%%%%%%%%%%%%%%%%%%%%%%%%%%%%%%%%%%%%%% 

	\begin{abstract}

In this paper, we address the stochastic reach-avoid problem for linear systems with additive stochastic uncertainty.
We seek to compute the maximum probability that the states remain in a safe set over a finite time horizon and reach a target set at the final time.
We employ sampling-based methods and provide a lower bound on the number of scenarios required to guarantee that our estimate provides an underapproximation.
Due to the probabilistic nature of the sampling-based methods, our underapproximation guarantee is probabilistic, and the proposed lower bound can be used to satisfy a prescribed probabilistic confidence level. 
To decrease the computational complexity, we propose a Voronoi partition-based to check the  reach-avoid constraints at representative partitions (cells), instead of the original scenarios. 
The state constraints arising from the safe and target sets are tightened appropriately so that the solution  provides an underapproximation for the original sampling-based method. 
We propose a systematic approach for selecting these representative cells and provide the flexibility to trade-off the number of cells needed for accuracy with the computational cost.      
	\end{abstract}
	
	%%%%%%%%%%%%%%%%%%%%%%%%%%%%%%%%%%%%%%%%%%%%%%%%%%%%% 
	%%%%%%%%%%%%%%%%%%%%%%%      Introduction    %%%%%%%%%%%%%%%%%%%%
	%%%%%%%%%%%%%%%%%%%%%%%%%%%%%%%%%%%%%%%%%%%%%%%%%%%%% 
	
	\section{Introduction} \label{sec:Intro}
	
    Reach-avoid analysis is an established verification tool for discrete-time stochastic dynamical systems, which provides probabilistic guarantees on the safety and performance~\cite{SummersAutomatica2010,HomChaudhuriACC2017,MaloneHSCC2014,lesser2013stochastic,GleasonCDC2017}. This paper focuses on the finite time horizon \emph{terminal} hitting
    time stochastic reach-avoid problem \cite{SummersAutomatica2010} (referred to here as the {\em terminal
    time problem}), that is, computation of the maximum probability
    of hitting a target set at the terminal time, while avoiding an unsafe set
    during all the preceding time steps using a controller that satisfies the
    specified control bounds.  
    
    The solution to the terminal time problem relies on dynamic programming~\cite{SummersAutomatica2010,AbateAutomatica2008,AbateHSCC2007}, hence a variety of approximation methods have been suggested in literature. Researchers have looked for scalable approaches to solve this problem using approximate dynamic programming \cite{KariotoglouSCL2016, ManganiniCYB2015}, Gaussian mixtures \cite{KariotoglouSCL2016}, particle filters \cite{ManganiniCYB2015, lesser2013stochastic}, convex chance-constrained optimization \cite{lesser2013stochastic}, Fourier transform-based verification~\cite{VinodLCSS2017, VinodHSCC2018}, Lagrangian approaches~\cite{GleasonCDC2017}, and semi-definite programming~\cite{SDP_kariotgolou}.  Currently, the largest system verified is a $40$-dimensional chain of double integrators~\cite{VinodLCSS2017, VinodHSCC2018} using Fourier transform-based techniques.  Existing methods impose a high computational complexity which makes them unrealistic for real-time applications.  

    In this paper, we reconsider the sampling-based approach, proposed in~\cite{lesser2013stochastic}. 
    Similar sampling-based approach has been used successfully in robotics~\cite{blackmore2010probabilistic} and in stochastic optimal control~\cite{satipizadeh2018Automatica,sartipizadeh2018CCTOOL,calafiore2013stochastic,calafiore2006scenario}. In the sampling-based stochastic reach-avoid problem, we sample the stochastic disturbance to produce a finite set of \emph{scenarios}, and then formulate a mixed-integer linear program (MILP) to maximize the number of scenarios that satisfy the reach-avoid constraints~\cite{blackmore2010probabilistic,lesser2013stochastic}. 
As expected, the approximated probability will converge to the true terminal time probability as the number of scenarios increases. However, the computational complexity of MILP increases exponentially with the number of binary decision variables (the number of scenarios)~\cite[Rem. 1]{bemporad_control_1999} making the MILP formulation practically intractable.  
 
 The main contributions of this paper are two-fold. We first provide a lower
    bound on the number of scenarios needed to probabilistically guarantee a user-specified
    upper bound on the approximation error with a user-specified confidence level using concentration techniques.
    Using Hoeffding's inequality, we demonstrate that the number
    of scenarios that need to be considered is inversely proportional
    to the square of the desired upper bound on the estimate error.
    Next, we propose a Voronoi-based undersampling technique that
    underapproximates the MILP-based solution in a computationally efficient manner. This approach allows us to partially mitigate the exponential computational complexity, and provides flexibility to select the number of partitions based on the allowable online computational complexity. We demonstrate the application of the proposed method in a problem of spacecraft rendezvous and docking.

    The organization of the paper is as follows:
Problem formulation and preliminary definitions are stated in Section~\ref{sec:ProblemDescription}. Lower bound on the required number of scenarios for the prescribed confidence level is given in Section~\ref{sec:BoundOnSamples}. Section~\ref{sec:Methodology} presents the proposed partition-based method and the approximate solution reconstruction. The performance of the proposed method is investigated on a spacecraft rendezvous maneuvering and docking in Section~\ref{sec:IllustrativeExample}.

	%%%%%%%%%%%%%%%%%%%%%%%%%%%%%%%%%%%%%%%%%%%%%%%%%%%%% 
	%%%%%%%%%%%%%%%%%%%%%%%      Problem Description    %%%%%%%%%%%%%%%
	%%%%%%%%%%%%%%%%%%%%%%%%%%%%%%%%%%%%%%%%%%%%%%%%%%%%% 
	
	\section{Problem formulation} 
	\label{sec:ProblemDescription}
We presume $\R$ and $\N$ are sets of real and natural numbers, with $\R^{n}$ a length $n$ vector of real numbers,
  and $\N_{[a,b]}$ the set of natural numbers between $a$ and $b$. For $x\in\R^{n}$, $x^{\top}$ denotes the transpose of $x$. {\color{black} Vector with all elements 1 is denoted $\mathbf{1}$.}

	\subsection{System description}
	
	Consider a discrete-time stochastic LTI system,
	\begin{equation} \label{eq:sys}
	x_{t+1}=Ax_t+Bu_t+w_t
	\end{equation}
    with state  $x_t\in \X=\R^{n_x}$,
    input $u_{t}\in \U\subseteq\R^{n_u}$, disturbance $w_t\in \W\subseteq\R^{n_x}$ at time instant $t$, and matrices $A,B$ assumed to be of appropriate dimensions.
	We assume that $w_k$ is an  independent and identically distributed (i.i.d) random variable with a PDF $\eta_{w}$.
	Note that we require $\eta_{w}$ only to be a probability density function from which we can draw samples, and do not require it to be Gaussian.
	The system~\eqref{eq:sys} over a time horizon with length $N$ can be alternatively written in a ``stacked'' form,
	\begin{equation} \label{eq:sys_stacked}
	X(x_{0},U,W)=G_{x}x_0+G_{u}U+G_{w}W,
	\end{equation}
	with $X=[x_{1}^{\top},\cdots,x_{N}^{\top}]^\top \in \X^N$,  $U=[u_{0}^{\top},\cdots,u_{N-1}^{\top}]^\top\in \U^N$, and $W=[w_{0}^{\top},\cdots,w_{N-1}^{\top}]^\top\in \W^N$ the concatenated state, input, and disturbance vectors over a $N$-length horizon \cite{sartipizadeh2018CCTOOL}. 
	The matrices $G_{x}$, $G_{u}$, and $G_{w}$ may be obtained from the system matrices in \eqref{eq:sys} (see~\cite{sartipizadeh2018CCTOOL}). 
	Due to the stochastic nature of $w_k$, the state $x_k$ and the concatenated state vector $X$ are random.
	We define $\Prob_{X}^{x_0,U}$ as the probability measure associated with the random vector $X$, which is induced from the probability measure of the concatenated disturbance vector $\Prob_W$ and \eqref{eq:sys_stacked}.
	By the i.i.d. assumption on $w_k$, $\Prob_W$ is characterized by $(\eta_{w})^N$.
	
	% subsection System formulation (end)
	
	\subsection{Stochastic reach-avoid problem}

	We are interested in the terminal time problem~\cite{SummersAutomatica2010}.
 As in \cite{SummersAutomatica2010}, we seek open-loop control laws, to assure tractability (at the cost of conservativeness~\cite{VinodLCSS2017,lesser2013stochastic,VinodHSCC2018}).
	We define the \emph{terminal time probability}, $\Costpi$, for a given initial state $x_0\in \X$ and an open-loop control $U\in \U^N$, as the probability that the state trajectory remains inside the safety set $\mcS \subseteq \X$ and reaches the target set $\mcT \subseteq \X$ at time $N$,
	\begin{align} 
	\Costpi &= \Prob_{X}^{x_0,U}\left\{x_{N}\in \mcT \wedge x_t\in \mcS,\forall t\in \mathbb{N}_{[0,N-1]}\right\}= \Prob_{X}^{x_0,U}\left\{ X\in \mcR\right\} 1_{\mcS}( x_0).\label{eq:r_prob} 
	\end{align}
	with $ \mcR = \mcS^{N-1}\times \mcT $.
	The stochastic reach-avoid problem is formulated as: 
	\begin{problem} \label{prb:Original}  Open-loop terminal time problem:
		\begin{maxi} |s|
			{U\in \U^{N}}{\Costpi}{\label{prob:original}} {p^{\ast}( x_{0})=}
		\end{maxi}
		Problem~\ref{prb:Original} is equivalent to (see~\cite[Sec. 4]{SummersAutomatica2010}),
		\begin{maxi} |s|
			{U\in \U^{N}}{1_\mcS( x_0)\Exp_z^{x_0,U}\left[  z \right],}{\label{prob:original_exp}}{p^{\ast}( x_{0})=}
		\end{maxi}
\hspace{-1.5mm}where $z=1_\mcT(x_N)\prod_{t=1}^{N-1}1_\mcS(x_t) = 1_{\mcR}(X)$ is a Bernoulli random variable with a discrete probability measure $\Prob_z^{x_0,U}$ induced from $\Prob_X^{x_0,U}$.
	\end{problem}
	\begin{remark}\label{rem:zeroProb}
		In Problem~\ref{prb:Original}, $p^\ast(x_0)$ is trivially zero when $x_0\not\in\mcS$, irrespective of the choice of the controller.
	\end{remark}

	In~\cite{lesser2013stochastic,blackmore2010probabilistic}, a mixed-integer linear program (MILP) was formulated as an approximation of Problem~\ref{prb:Original} when the safe and the target sets are \emph{polytopic}.
    Note that restriction of the safe and the target sets to polytopes is not severe since convex and compact sets admit tight polytopic underapproximations~\cite[Ex. 2.25]{BoydConvex2004}.
	We will denote the safe set $\mcS$, the target set $\mcT$, and the reach-avoid constraint set $ \mcR$ as
    \begin{subequations}
	\begin{align}
	\mcS&=\{x|f_{\mcS}x\leq h_{\mcS}\}, \\
	\mcT&=\{x|f_{\mcT}x\leq h_{\mcT}\}, \\
    \mcR&=\{x|FX\leq h\}\label{eq:mcR_defn}
	\end{align}%
    \end{subequations} %
	with $l_\mcS,l_\mcT\in \N, L=(N-1)l_\mcS+l_\mcT, f_\mcS\in \mathbb{R}^{l_\mcS\times n_x}$, $f_\mcT\in \mathbb{R}^{l_\mcT\times n_x}$, and $F\in \mathbb{R}^{L\times n_x}$ is constructed using $f_{\mcS}$ and $f_{\mcT}$.
	Using the ``big-M" approach~\cite{bemporad_control_1999,blackmore2010probabilistic,lesser2013stochastic}, and a sampling-based empirical mean for $\Costpi$ that replaces $\W^{N}$ by a finite set of $K$ random samples, 
	\begin{equation}
	\W_{K}=\{W^{(1)},\cdots,W^{(K)}\},
	\end{equation}
	 we obtain a MILP approximation to Problem~\ref{prb:Original}.
	\begin{problem} \label{prb:MILP_SamplingBased} A MILP approximation to Problem~\ref{prb:Original} is given by
		\begin{maxi*}|s|
			{\substack{U\in\U^{N}}}{\frac{1}{K}\sum_{i=1}^{K}z^{(i)}}{}{}%p_{K}^{\ast}(x_0)=}
			\addConstraint{X^{(i)}}{= G_{x}x_{0} + G_{u} U + G_{w}W^{(i)},}{\ i\in \N_{[1,K]}}
			\addConstraint{FX^{(i)}}{\leq h + M(1-z^{(i)})\mathbf{1},}{\ i\in \N_{[1,K]}}
			\addConstraint{z^{(i)}}{\in\{0,1\},}{\ i\in \N_{[1,K]}}
\end{maxi*}
\hspace{-1.5mm}with the optimal value denoted by $p_{K}^{\ast}(x_0)$, optimal control input $U^{\ast}_{K}$, $M \in \mathbb{R}$  some large positive number, and $W^{(i)}$ that are concatenated disturbance realizations sampled from $ \W^N$, based on the probability law $\Prob_W$.
	\end{problem}
	As observed in~\cite{lesser2013stochastic,bemporad_control_1999,blackmore2010probabilistic}, $z^{(i)}$ takes the value $1$ if and only if $FX^{(i)}\leq h$ for all $i\in \mathbb{N}_{[1,K]}$.
	For any sampled trajectory that violates the reach-avoid constraint $X^{(j)}\in \mcR,\ j\in \mathbb{N}_{[1,K]}$, we have $FX^{(j)}> h$ which is encoded by $z^{(j)}=0$ . This concept is illustrated in Figure~\ref{fig:voronoi_0}; red crosses indicate the sampled trajectories which fail to remain in $\mcR$ and, therefore, their corresponding binary variables are zero.
	From~\cite{lesser2013stochastic}, we have,
	\begin{align}
	\lim_{{K \rightarrow \infty}}p_{K}^{\ast}(x_0) = p^{\ast}(x_0).
	\end{align}
	However, Problem~\ref{prb:MILP_SamplingBased} becomes computationally intractable for large values of $K$ since the worst-case time complexity of MILP problems exponentially increases in the number of binary decision variables~\cite[Rem. 1]{bemporad_control_1999}.
	
	% subsection Stochastic reach-avoid problem (end)
	\begin{figure}
		\centering
		\includegraphics[width=2.6in]{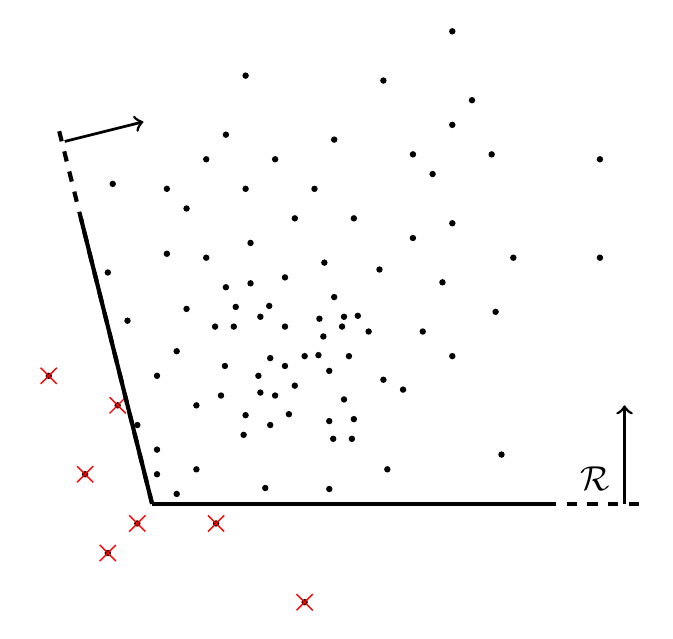}
		\vspace{-2mm}
        \caption{Sampling-based approach illustration for reach-avoid problem in $2$D. Reach-avoid set $\mcR$ is distinguished with lines. Black dots and red crosses represent the sampled state trajectories corresponding to sampled disturbance $\W_{K}$ for a given $U$ and $x_0$ that succeed and fail to remain in $\mcR$, respectively. Empirical mean of remaining in reach-avoid set is obtained by dividing the number of samples inside $\mcR$ to the total number of samples. }
		\label{fig:voronoi_0}
	\end{figure}
	
	\subsection{Problem statements}

	Based on Problem~\ref{prb:MILP_SamplingBased}, 
we define the random vector $Z={[z^{(1)}\ \ldots\ z^{(K)}]}^\top$, the concatenation of an i.i.d process consisting of Bernoulli random variables  $\{z^{(i)}\}_{i=1}^K$.
	By definition, the probability measure associated with $Z$ is $\Prob_{Z}^{x_0, U}=\prod_{i=1}^K\Prob_z^{x_0,U}$.
	
	We will address the following questions:
	\begin{question}\label{ques:Bound}
{\color{black}Given a violation parameter $\delta\in[0,1]$ and a risk of failure $\beta\in[0,1]$,  characterize the sufficient number of scenarios $K$ to guarantee $\Prob_{Z}^{x_0, U^\ast_K}\{ p_{K}^\ast(x_0) - p^{\ast}(x_0) \geq \delta\}\leq \beta$ or equivalently $\Prob_{Z}^{x_0, U^\ast_K}\{ p^{\ast}(x_0) \geq p_{K}^\ast(x_0)-\delta \}\geq 1-\beta$, for all $x_0\in \mcS$.}

	\end{question}
	\begin{question}\label{ques:Voronoi}
		{\color{black}Given $K$ scenarios as characterized in Question~\ref{ques:Bound} to meet desired specifications, construct an under-approximate MILP with $\hat{K}<K$ scenarios (hence binary decision variables) that results in the terminal time probability estimate $\hat{p}(x_0)$ with $\hat{p}(x_0)\leq p_{K}^{\ast}(x_0)$. }
	\end{question}

    We will address Question~\ref{ques:Bound} using Hoeffding's inequality. By solving Question~1, we seek sufficient number of  scenarios that leads to a desired upper bound on the likelihood that the approximate solution exceeds the true solution by some threshold ($\delta$). 
{\color{black}Note that a smaller $\delta$ implies less conservatism as well as higher accuracy in estimation, but requires more scenarios for a fixed $\beta$, as claimed in next section. Then we address Question~\ref{ques:Voronoi} using Voronoi partitions to reduce the number of scenarios while preserving the original specifications $\delta$ and $\beta$.}
%%
%% Voronoi partition and data clustering
%%
\subsection{Voronoi partition and data clustering}
	\label{sec:VoronoiClustering}
{\color{black}Here we introduce some preliminaries on Voronoi partition that we will use in the rest of this paper. }
Given a set of seeds (centres) $\mcC=\left\{c^{(1)},\cdots,c^{(\hat{K})}\right\}$, $c^{(i)}\in\R^{d}$,
    a Voronoi partition $\mcV(\mcC)$ partitions the $\R^d$ space to $\hat{K}$
    cells $V^{(1)},\cdots,V^{(\hat{K})}$ such that any point in $V^{(j)}$,
    $\forall j\in\N_{[1,\hat{K}]}$, is closer to $c^{(j)}$ than the other
    seeds. Given a set of points $\mcP=\left\{p^{(1)},\cdots,p^{(K)}\right\}$ in  $\R^d$, we use $\mcV_{\mcP}(\mcC)$ to show the partition of $\mcP$ through a Voronoi partition with seeds $\mcC$. We define each cell (may also be referred to as partition in this paper) of $\mcV_{\mcP}(\mcC)$ as
	\begin{align}
        V^{(j)}_{\mcP}(\mcC)&=\{p\in \mcP|d(p,c^{(j)})\leq
            d(p,c^{(\ell)}), \forall j,\ell\in\N_{[1,\hat{K}]} ~~\text{and}~~ j\neq\ell   \},\label{eq:V_j}
	\end{align}
	where $d(p^{(1)},p^{(2)})$ is the distance of $p^{(1)}\in \R^d$ from $p^{(2)}\in \R^d$ in any valid metric (Euclidean norm is used in this paper). 
    We denote the number of elements in $V^{(j)}$ by $ \vert V^{(j)} \vert$.  

    A given set $\mcP\in \mathcal{X}^K$ with $K$ points in $\R^d$ can be clustered in $\hat{K}$ clusters by finding a set of seeds $\mcC\in \mathcal{X}^{\hat{K}}$ that minimizes the \emph{within-cluster sum of squares}, as proposed by $k$-means method/Lloyd's algorithm~\cite{hartigan1975clustering}. The within-cluster sum of squares, denoted by $\mathrm{WSS}(\hat{K},\mcV_{\mcP}(\mcC))$, simply represents the total sum of squared deviations of points in cells from their seeds, i.e., given a set of points $\mcP$, a set of seeds $\mcC$, and a partition set $\mcV_{\mcP}(\mcC)=\{V^{(1)},\cdots,V^{(\hat{K})}\}$,
\begin{equation} \label{eq:wss}
    \mathrm{WSS}(\hat{K},\mcV_{\mcP}(\mcC))=\sum_{j=1}^{\hat{K}}\sum_{i=1}^{K}  1_{V^{(j)}}(p^{(i)})\|p^{(i)}-c^{(j)}\|^2.
\end{equation}
where $1_{V^{(j)}}(p^{(i)})$ is an indicator function which is one if $p^{(i)}$ belongs to cell $V^{(j)}$, and zero otherwise. Let 
\begin{equation} \label{eq:Cstar}
    \mcC^\ast=\underset{\mcC\in \mathcal{X}^{\hat{K}}}{\operatorname{arg\ min}}\ \mathrm{WSS}(\hat{K},\mcV_{\mcP}(\mcC)),
\end{equation}
be the set of optimal seeds. Given $\mcC^\ast$, cells of $\mcV_{\mcP}(\mcC^\ast)$ represent the optimal $\hat{K}$ clusters of $\mcP$. Although solving \eqref{eq:Cstar} is an NP-hard problem \cite{Aloise2009}, efficient algorithms exist to compute a local minima. Starting from an initial guess of $\mcC$, a successive algorithm with time complexity $\mathcal{O}(ndK\hat{K})$ for $n$ iterations can be used to update each seed by replacing it with the centroid of its cluster elements \cite{hartigan1979algorithm}. This process continues until convergence or $n$ reaches its maximum value. Note that the set of optimal seeds $\mcC^\ast$ is not necessarily a subset of $\mcP$. In addition, $\mathrm{WSS}(\hat{K},\mcV_{\mcP}(\mcC))$ is non-increasing in $\hat{K}$, and the number of required clusters can be decided by making a trade-off between $\mathrm{WSS}(\hat{K},\mcV_{\mcP}(\mcC))$ (representing accuracy) and $\hat{K}$ (representing computational complexity).       
	
	\begin{lemma} \label{lem:translation}
		Seed selection and Voronoi configuration are preserved under  translation. 
		\begin{enumerate}
			\item Let $\mcC^\ast=\{c^{(1)},\cdots,c^{(\hat{K})}\}$ be the set of $\hat{K}$ optimal seeds of $\mcP=\{p^{(1)},\cdots,p^{(K)}\}$, calculated as in \eqref{eq:Cstar}. For any fixed vector $a$, $\mcC^{\ast}_{a}=\{a+c^{(1)},\cdots,a+c^{(\hat{K})}\}$ represents the optimal seeds of $\mcP_{a}=\{a+p^{(1)},\cdots,a+p^{(K)}\}$.
			\item For any $p^{(i)}\in\mcP$, let $p^{(i)}\in V^{(j)}_{\mcP}(\mcC^{\ast})$. Then $a+p^{(i)}\in V^{(j)}_{\mcP_a}(\mcC_a^{\ast})$.
		\end{enumerate}
	\end{lemma}
	\begin{proof}
		The proof is straight forward since both $\mathrm{WSS}$ and Voronoi partitions (as given in \eqref{eq:wss} and \eqref{eq:V_j}) are functions of the relative distance of the points in each cell to their seed. 
	\end{proof}
	% subsection Problem statements (end)

	\section{Scenarios required to meet given failure tolerance}
	\label{sec:BoundOnSamples}
	Given i.i.d. $y^{(1)}, y^{(2)},\ldots,y^{(K)}$ for $K>0$ and $y^{(i)}\in[0,1],\ \forall i\in \mathbb{N}_{[1,K]}$ with probability measure $\Prob_y$, we define the concatenation of these random variables $Y=[y^{(1)}\ y^{(2)}\ \ldots\ y^{(K)}]^\top\in {[0,1]}^K$.
	The probability measure associated with $Y$ is $\Prob_Y^K=\prod_{i=1}^K \Prob_y$.
	We have the following inequalities from Hoeffding~\cite[Thm.
    1]{hoeffding_probability_1963}.

\begin{lemma}{\textbf{(Hoeffding's inequality)}}\label{lem:Hoeffding}
		Define $ \overline{Y}=\frac{\mathbf{1}^\top Y}{K}=\frac{\sum_{i=1}^K y^{(i)}}{K}$, and $\mu_{\overline{Y}}\triangleq\Exp\left[ \overline{Y}  \right]$.
		For any $\delta>0$,
		\begin{align}   \label{eq:Hoeff_OneSide} 
		\Prob_{Y}^K\left\{ \overline{Y} - \mu_{\overline{Y}} \geq \delta \right\}&\leq e^{-2K\delta^2}.    %\\
		\end{align}
	\end{lemma}
	
	\begin{theorem}\label{thm:K_lb}
	{\color{black}	Given a violation parameter $\delta\in[0,1]$,} risk of failure $\beta\in[0,1]$, initial state $x_0\in \mcS$, and the optimal solution $U^\ast_K\in \mathcal{U}^N$ to Problem~\ref{prb:MILP_SamplingBased}, we have the risk of failure $\Prob_{Z}^{x_0, U^\ast_K}\{ p^\ast(x_0) - p_{K}^{\ast}(x_0) \geq \delta\}\leq \beta$, if 
		\begin{align}
		K\geq \frac{-\ln(\beta)}{2\delta^2}.\label{eq:K_lb}
		\end{align}
	\end{theorem}
	\begin{proof}
		Let the optimal solution to Problem~\ref{prb:Original} be $U^\ast\in \mathcal{U}^N$ (which may not be equal to $U^\ast_K$). For $x_0\in \mcS$,
		\begin{subequations}
			\begin{align}
			p^\ast(x_0)&= \Exp_z^{x_0,U^\ast}\left[  z \right]=\frac{\Exp_Z^{x_0,U^\ast}\left[  \mathbf{1}^\top Z \right]}{K}, \label{eq:milp_p_opt_defn_past}\\
			p^\ast_K(x_0)&=\frac{1}{K}\sum_{i=1}^K z^{(i)}=\frac{\mathbf{1}^\top Z}{K}\mbox{ under } U^\ast_K, \label{eq:milp_p_opt_defn_pastk}\\
			\Exp_z^{x_0,U^\ast}\left[  z \right] &\geq \Exp_z^{x_0,U^\ast_K}\left[  z \right]\label{eq:milp_p_opt_defn_exp}
			\end{align}\label{eq:milp_p_opt_defn}%
		\end{subequations}%
		Using \eqref{eq:milp_p_opt_defn_past} and \eqref{eq:milp_p_opt_defn_pastk},
		\begin{align}
		\Prob_{Z}^{x_0, U^\ast_K}\{ p_{K}^{\ast}(x_0)-p^\ast(x_0) \geq \delta\} &=\Prob_{Z}^{x_0, U^\ast_K}\left\{\left( \frac{1}{K}\sum_{i=1}^K z^{(i)} - \Exp_z^{x_0,U^\ast}\left[  z \right] \right) \geq \delta\right\}.\label{eq:prob_ineq}
		\end{align}
		Adding and subtracting $\Exp_z^{x_0,U^\ast_K}\left[  z \right]$, we have
		\begin{align}
		\left( \frac{1}{K}\sum_{i=1}^K z^{(i)} - \Exp_z^{x_0,U^\ast}\left[  z \right] \right) &\ = \left( \frac{1}{K}\sum_{i=1}^K z^{(i)} - \Exp_z^{x_0,U^\ast_K}\left[  z \right] \right) +\left( \Exp_z^{x_0,U^\ast_K}\left[  z \right] - \Exp_z^{x_0,U^\ast}\left[  z \right] \right)\nonumber \\
		&\ \leq \left( \frac{1}{K}\sum_{i=1}^K z^{(i)} - \Exp_z^{x_0,U^\ast_K}\left[  z \right] \right)\label{eq:ineq_exp} 
		\end{align}
		where \eqref{eq:ineq_exp} follows from \eqref{eq:milp_p_opt_defn_exp}.
		Thus, $\left\{ \overline{Z}\in\{0,1\}^K: \left( \frac{\mathbf{1}^\top
        \overline{Z}}{K} - \Exp_z^{x_0,U^\ast}\left[  z \right]\right)\geq
    \delta\right\}$ is  a subset of\\$\left\{\overline{Z}\in\{0,1\}^K: \left( \frac{\mathbf{1}^\top \overline{Z}}{K}  - \Exp_z^{x_0,U^\ast_K}\left[  z \right]\right)\geq \delta\right\} $ by \eqref{eq:milp_p_opt_defn_pastk} and \eqref{eq:ineq_exp}.
		Using the fact that $\Prob\{ \mathcal{S}_1\}\leq \Prob\{ \mathcal{S}_2\}$ for any two sets $ \mathcal{S}_1\subseteq \mathcal{S}_2$, Hoeffding's inequality \eqref{eq:Hoeff_OneSide}, $\Exp_z^{x_0,U^\ast_K}\left[  z \right]=\frac{\Exp_Z^{x_0,U^\ast_K}\left[  \mathbf{1}^\top Z \right]}{K}$, and \eqref{eq:prob_ineq}, we have
		\begin{align}
		&\Prob_{Z}^{x_0, U^\ast_K}\{ p_{K}^{\ast}(x_0)-p^\ast(x_0) \geq \delta\} \leq\Prob_{Z}^{x_0, U^\ast_K}\left\{\left( \frac{\mathbf{1}^\top Z}{K} - \frac{\Exp_Z^{x_0,U^\ast_K}\left[  \mathbf{1}^\top Z \right]}{K} \right) \geq \delta\right\}\leq e^{-2K\delta^2}. \nonumber
		\end{align}
		To obtain the desired probabilistic guarantee $\Prob_{Z}^{x_0, U^\ast_K}\{ p_{K}^{\ast}(x_0)-p^\ast(x_0) \geq \delta\} \leq\beta$, we require $e^{-2K\delta^2}\leq \beta$.
		Solving for $K$, we obtain $K\geq \frac{-\ln(\beta)}{2\delta^2}$.
	\end{proof}
	
	Theorem~\ref{thm:K_lb} addresses Question~\ref{ques:Bound}.
	Specifically,   choosing at least $K$  scenarios, Theorem~\ref{thm:K_lb} guarantees that the probability of the event that the MILP-based estimated terminal  time probability (the optimal solution to Problem~\ref{prb:MILP_SamplingBased}) exceeds the true terminal time probability (the optimal solution to Problem~\ref{prb:Original}) by more than $\delta$ is less than $\beta$ (a small value). Here, both $\delta$ and $\beta$ are provided by the user. Note that although $p^{\ast}$ and $p_{K}^{\ast}$ are functions of the time horizon $N$, $K$ is independent of the choice of $N$.
    A similar bound is used for the application of aircraft conflict detection in~\cite{prandini2000probabilistic}.
	
	%%%%%%%%%%%%%%%%%%%%%%%%%%%%%%%%%%%%%%%%%%%%%%%%%%%%% 
	%%%%%%%%%%%%%%%%%%%%%%%      Methodology    %%%%%%%%%%%%%%%%%%%
	%%%%%%%%%%%%%%%%%%%%%%%%%%%%%%%%%%%%%%%%%%%%%%%%%%%%% 
	
	\section{Partition-based sample reduction}  
	\label{sec:Methodology}
	
	As implied from the concentration probability bounds given in Theorem~\ref{thm:K_lb}, Problem~\ref{prb:MILP_SamplingBased} typically needs a large number of samples to provide a precise approximation for Problem~\ref{prb:Original} with a small deviation $\delta$ and small risk of failure $\beta$. Therefore, solving Problem~\ref{prb:MILP_SamplingBased} can be computationally expensive or even intractable for real-time applications. In this section, we address Question~\ref{ques:Voronoi} by proposing a \emph{partition-based} method which provides an underapproximation to Problem~\ref{prb:MILP_SamplingBased} with flexible computational complexity, as opposed to  the sampling-based approach, presented in Problem~\ref{prb:MILP_SamplingBased}. To this end, we propose the following MILP problem with $\hat{K}$ binary variables, where $\hat{K}$ can be significantly smaller than $K$ and is selected by the user. 
	
	\begin{problem} \label{prb:MILP_truncated} The partition-based terminal time problem is
 \vspace{-3mm}   
		\begin{maxi*}|s|
			{\substack{U\in\U^N 
					}}{\frac{1}{K}\sum_{j=1}^{\hat{K}}\alpha^{(j)}\hat{z}^{(j)}}{}{}%p_{K}^{\ast}(x_0)=}
			\addConstraint{\hat{X}^{(j)}}{= G_{x}x_{0} + G_{u} U + \psi^{(j)},}{\ j\in \N_{[1,\hat{K}]}}
			\addConstraint{F\hat{X}^{(j)}}{\leq h -\varepsilon^{(j)}+ M(1-\hat{z}^{(j)})\mathbf{1},}{\ j\in \N_{[1,\hat{K}]}}
			\addConstraint{\hat{z}^{(j)}}{\in\{0,1\},}{\ j\in \N_{[1,\hat{K}]}}	\end{maxi*}
\hspace{-2mm} with the optimal value denoted by $p_{\hat{K}}^{\ast}(x_0)$. 
        Here, $M \in \mathbb{R}$ is some large positive number, and $\psi^{(j)}$ for $j\in\N_{[1,\hat{K}]}$ are $\hat{K}$ selected representatives  (seeds) of uncertainty, computed in a prediction mapping $\phi:\W^{N}\rightarrow \X^{N}$ with $$\phi(W):=G_{w}W.$$ $\alpha^{(j)}$ is the importance rate of the $j^{th}$ seed with $\sum_{j=1}^{\hat{K}}\alpha^{(j)}=K$ and $\alpha^{(j)}\in\N_{[1,K]}$. For $j\in\N_{[1,\hat{K}]}$, $\varepsilon^{(j)}$ is an appropriately designed buffer that guarantees the solution of  Problem~\ref{prb:MILP_truncated} is a lower bound on the solution of Problem~\ref{prb:MILP_SamplingBased}.
	\end{problem}
	
	In Problem~\ref{prb:MILP_truncated}, the state uncertainty is characterized by $\hat{K}$ seeds where the $j^{th}$ seed represents $\alpha^{(j)}$ scenarios of $\W_{K}$. Then the reach-avoid constraints are only checked at the selected seeds instead of being checked at every scenario, which reduces the number of binary variables and constraints. 
	
	\subsection{Seed Selection and Buffer Computation}
	
    Given a sample set $\W_K$, $x_{0}\in\mcS$, and $U$, we define $\X_{K}^{x_0,U}:=X(x_0,U,\W_K)$ as the set of sampled state trajectories.
    We desire that the elements $\X_{K}^{x_0,U}$ remain in reach-avoid set $\mcR$. 
	The set $\X_{K}^{x_0,U}$ can be partitioned into cells, where each cell consists of some of the random state trajectories and is represented by a seed. Figure~\ref{fig:voronoi_partinion} shows a $2$D partition with 11 seeds. 
    
	\begin{figure}
		\centering
		\includegraphics[width=3in]{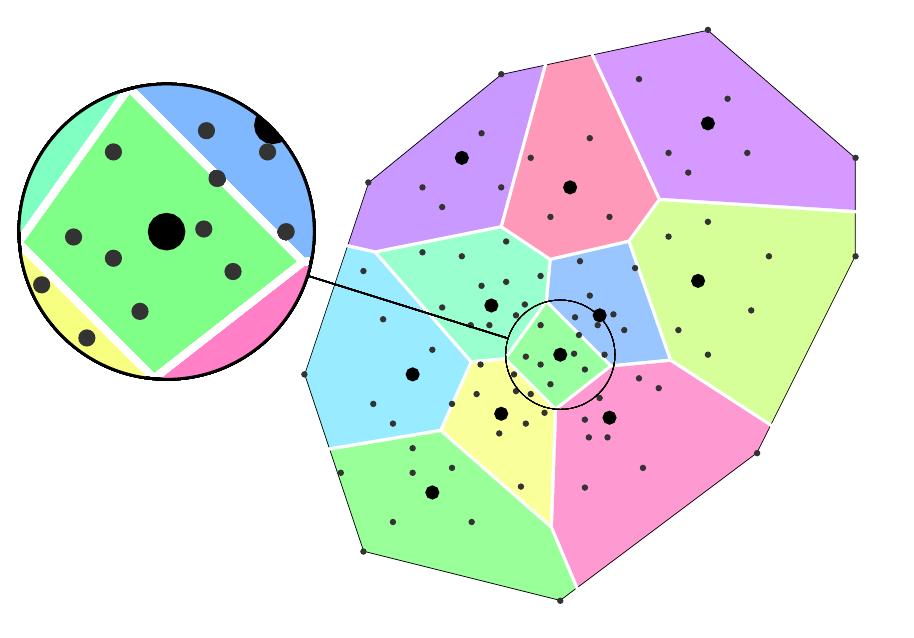}
		\caption{Partitioning the state uncertainty region of Figure~\ref{fig:voronoi_0} to $\hat{K}$ cells using a Voronoi partition. Larger dots indicate the selected Voronoi seeds. Samples  inside each cell are closer to their own seed than  other seeds.}
		\label{fig:voronoi_partinion}
	\end{figure}
    
	\begin{lemma} \label{lem:Voronoi}
		Let $\Psi_{\hat{K}}:=\{\psi^{(1)},\cdots,\psi^{(\hat{K})}\}$ be the set of optimal seeds of $\Phi_{K}:=\phi(\W_{K})$ with $\phi(W):=G_{w}W$ that minimizes $\mathrm{WSS}$. Then $\hat{\X}_{\hat{K}}^{x_0,U}=\hat{X}(\Psi_{\hat{K}})=\{\hat{X}(\psi^{(1)}),\cdots,\hat{X}(\psi^{(\hat{K})})\}$ with 
		\[
		\hat{X}(\psi^{(j)})=G_{x}x_{0}+G_{u}U+\psi^{(j)},
		\]
		represents the set of optimal seeds for sampled state trajectory set $\X_{K}^{x_0,U}$. 
	\end{lemma}
	\begin{proof}
		Results directly from Lemma~\ref{lem:translation}. Since there is no uncertainty in $G_x$ and $G_u$, $G_{x}x_0+G_{u}U$ in \eqref{eq:sys_stacked} can be interpreted as a translation term. 
	\end{proof}
	
	According to Lemma~\ref{lem:Voronoi}, although the state trajectory is an optimization variable, it can be clustered through the prediction mapping $\phi(W)$ \emph{offline},  independent of the choice of $x_0$ and $U$. 
	
	\begin{lemma}  \label{lem:buffer}
		Let a set of points $\Phi_{K}$, a set of selected seeds $\Psi_{\hat{K}}$, as defined in Lemma~\ref{lem:Voronoi}, and their Voronoi partition $\mcV_{\Phi_{K}}(\Psi_{\hat{K}})$ with  cells $V_{\Phi_K}^{(1)}(\Psi_{\hat{K}}),\cdots,V_{\Phi_K}^{(\hat{K})}(\Psi_{\hat{K}})$ be given. Then, for $j\in\N_{[1,\hat{K}]}$, every point $\phi\in V_{\Phi_K}^{(j)}(\Psi_{\hat{K}})$ remains in the original constraint set $FX\leq h$ if  $\psi^{(j)}$, the $j^{th}$ seed, remains in the buffered constraint set $F\hat{X}(\psi^{(j)})\leq h-\varepsilon^{(j)}$ with   
		\begin{equation} \label{eq:varepsilon}
		\varepsilon^{(j)}=\left[{\epsilon_{1}^{(j)}},\cdots,{\epsilon_{L}^{(j)}}\right]^{\top},
		\end{equation}
		and
		\begin{equation} \label{eq:epsilon}
		\epsilon_{\ell}^{(j)}:=\max_{\phi\in V_{\Phi_K}^{(j)}(\Psi_{\hat{K}})} \left(F_{\ell}\phi-F_{\ell}\psi^{(j)}\right), \hspace{5mm} \ell\in\N_{[1,L]}.
		\end{equation} 
\end{lemma}
	\begin{proof}
From~\eqref{eq:epsilon} we conclude that for all $\ell\in\N_{[1,L]}$, $j\in\N_{[1,\hat{K}]}$, 
\begin{equation} \label{eq:L4_P1}
F_{\ell}\phi\leq F_{\ell}\psi^{(j)}+\epsilon_{\ell}^{(j)} \hspace{10mm} \forall \phi\in V_{\Phi_K}^{(j)}(\Psi_{\hat{K}}).
\end{equation}
By adding $F_{\ell}\left(G_{x}x_0+G_{u}U\right)$ to the right and left sides of \eqref{eq:L4_P1}, we have
\begin{equation} \label{eq:L4_P2}
F_{\ell}\left(G_{x}x_0+G_{u}U+\phi\right)\leq F_{\ell}\left(G_{x}x_0+G_{u}U+\psi^{(j)}\right)+\epsilon_{\ell}^{(j)}. 
\end{equation}
Consequently, for all $\phi\in V_{\Phi_K}^{(j)}(\Psi_{\hat{K}})$, the following holds for any initial state and input trajectory, 
\begin{equation} \label{eq:L4_P3}
F_{\ell}X(x_0,U,\phi)\leq F_{\ell}\hat{X}(x_{0},U,\psi^{(j)})+\epsilon_{\ell}^{(j)}.
\end{equation}
Denoting $F_{\ell}X(x_0,U,\phi)$ and $F_{\ell}\hat{X}(x_0,U,\psi^{(j)})$ with $F_{\ell}X(\phi)$ and $F_{\ell}\hat{X}(\psi^{(j)})$, when $F_{\ell}\hat{X}(\psi^{(j)})\leq h_{\ell}-\epsilon_{\ell}^{(j)}$, it is concluded from \eqref{eq:L4_P3} that $F_{\ell}X(\phi)\leq h_{\ell}$, $\forall \phi\in V_{\Phi_K}^{(j)}(\Psi_{\hat{K}})$.  
Since \eqref{eq:L4_P3} is valid for all $\ell\in\N_{[1,L]}$ and $j\in\N_{[1,\hat{K}]}$, the proof is completed. 
\end{proof}

The buffering concept is illustrated in Figure~\ref{fig:voronoi_buffer}.
  	
	\begin{figure}
		\centering
		\includegraphics[width=2.5in]{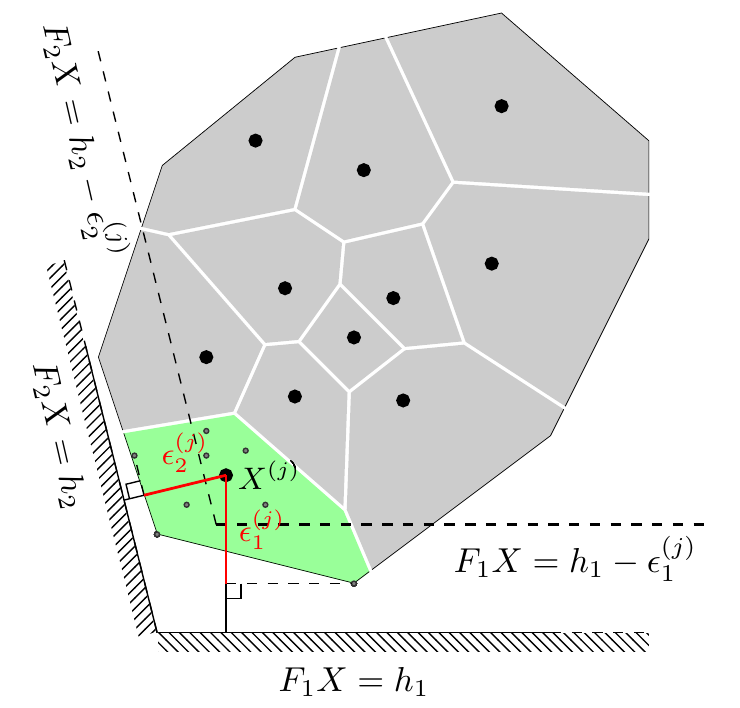}
		\vspace{-3mm}
        \caption{Buffering process. Cell $V^{(j)}$ is shown in green. If $X(\psi^{(j)})$, the state trajectory corresponding to the seed of $V^{(j)}$, remains in the buffered constraint $F_{\ell}X^{(j)}\leq h_{\ell}-\epsilon_{\ell}^{(j)}$, the state trajectory of every sample in $V^{(j)}$ will satisfy the original constraint $F_{\ell}X\leq h_{\ell}$.}
		\label{fig:voronoi_buffer}
	\end{figure}
    
	\begin{remark} 
        Computing $\epsilon_{\ell}^{(j)}$ for $\ell=1,\cdots,L$ and  $j\in \N_{[1,\hat{K}]}$ is a sorting problem in $1$D  and can be executed by worst time complexity of $\mathcal{O}(\sum_{j=1}^{\hat{K}}\alpha^{(j)} \log \alpha^{(j)})$.
	\end{remark} 
	\begin{theorem} \label{thm:MILP_truncated}
		Let $\Phi_{K}$ be a set of $K$ disturbance samples mapped through the prediction mapping $\phi(W):=G_{w}W$ and let $\Psi_{\hat{K}}$ be a set of selected seeds. Let $\alpha^{(j)}=\vert V_{\Phi_{K}}^{(j)}(\Psi_{\hat{K}})\vert$, $j\in\N_{[1,\hat{K}]}$, denote the number of elements of $V_{\Phi_{K}}^{(j)}(\Psi_{\hat{K}})$ and define $\varepsilon^{(j)}$ as in  Lemma~\ref{lem:buffer}. Problem~\ref{prb:MILP_truncated} provides a lower bound for Problem~\ref{prb:MILP_SamplingBased}.
	\end{theorem}
	
	\begin{proof}
		According to the definition of the buffers, if seed $\hat{X}(\psi^{(j)})$, $\forall j\in\N_{[1,\hat{K}]}$,  remains in the buffered constraint set, all $\alpha^{(j)}$ points of cell $V_{\Phi_{K}}^{(j)}(\Psi_{\hat{K}})$ remain in the original constraint set. Otherwise,  at most $\alpha^{(j)}$ samples belonging to cell $V_{\Phi_{K}}^{(j)}(\Psi_{\hat{K}})$ may violate the original constraints. Since in Problem~\ref{prb:MILP_truncated} the worst case is considered by weighting the $j^{th}$ seed with $\alpha^{(j)}$,  $p^{*}_{\hat{K}}$ provides a lower bound on $p^{*}_{K}$.       
	\end{proof}
	\begin{remark}
		If initial state $x_{0}\in\mcS$ is uncertain, the proposed method can be applied by defining $\phi(x_0,W):=G_{x} x_{0}+G_{w}W$.
	\end{remark}
	
	Obviously, having more cells results in a higher accuracy and when number of cells tends to the number of samples, $p^{*}_{\hat{K}}$ tends to $p^{*}_{K}$. 
    However, this improved accuracy comes at a higher computational cost. 
    We thus have to select $\hat{K}$ by trading off accuracy and computational cost.

	%%%%%%%%%%%%%%%%%%%%%%%%%%%%%%%%%%%%%%%%%%%%%%%%%%%%% 
	%%%%%%%% probability reconstruction %%%%%%%%%%%%%%%%%
	%%%%%%%%%%%%%%%%%%%%%%%%%%%%%%%%%%%%%%%%%%%%%%%%%%%%% 

	\subsection{Tightening the Voronoi-based terminal time probability estimate}
	\label{sec:Reconstruction}
	
Given $U_{\hat{K}}^{*}$ obtained from Problem~\ref{prb:MILP_truncated}, a tighter underapproximation on $p^{*}_{K}$ can be recalculated by simply checking the percentage of the original $K$ sampled trajectories that remain in reach-avoid set $\mcR$,  after applying $U_{\hat{K}}^{*}$ to the stochastic system.
    In contrast to solving a large MILP (as done in Problem~\ref{prb:MILP_SamplingBased}) whose computational complexity grows exponentially with $K$, this improved estimate (see \eqref{eq:phat}) is obtained by a policy evaluation that has a computational complexity of $ \mathcal{O}(K)$.
    The following theorem presents the probability underapproximation proposed in this paper. 
	\begin{theorem} \label{thm:phat}
		Let $p_{K}^{*}$ and $p_{\hat{K}}^{*}$ be the optimal values of Problem~\ref{prb:MILP_SamplingBased} and Problem~\ref{prb:MILP_truncated}, respectively, with corresponding optimal solutions $U_{K}^{*}$ and $U_{\hat{K}}^{*}$. Define $\hat{p}$ as
		\begin{equation} \label{eq:phat}
        \hat{p}=\frac{1}{K}\sum_{i\in\N_{[1,K]}} 1_{\mcR}\left(X(x_{0},U^{*}_{\hat{K}},W^{(i)})\right).
		\end{equation}
        Then 
		\begin{equation*}
		p_{\hat{K}}^{*}\leq \hat{p}\leq p_{K}^{*}.
		\end{equation*}
	\end{theorem}
	
	\begin{proof}
        i) Since $p^{*}_{K}$ is the optimal terminal time probability with $K$ samples and $\hat{p}$ is the evaluation of an open-loop controller $U_{\hat{K}}^\ast$ over these $K$ samples, we conclude that $\hat{p}\leq p_{K}^{*}$. Equality holds if $U_{\hat{K}}^{*}=U_{K}^{*}$. 
		
        ii) Let $\mathcal{J}=\{j\in\N_{[1,\hat{K}]}|\hat{z}^{(j)}=1\}$, the subset of $\mcC^\ast$ which were deemed safe by Problem~\ref{prb:MILP_truncated}.
        By definition of $\alpha^{(j)}$, $\sum_{\{i\in \N_{[1,K]}: W^{(i)}\in V^{(j)}\}} 1_{\mcR}\left(X(x_{0},U^{*}_{\hat{K}},W^{(i)})\right) = \sum_{\{i\in \N_{[1,K]}: W^{(i)}\in V^{(j)}\}} 1$ for every $j\in \mathcal{J}$.
        In other words, since $\alpha^{(j)}$ is the set of original scenarios that fall in the $j^\mathrm{th}$ cell, whenever the solution of Problem~\ref{prb:MILP_truncated} deems the representative seed safe, all the scenarios within it are safe.
        Thus, we have $\hat{p}$ at least as big as $\frac{1}{K}\sum_{j\in \mathcal{J}}\alpha^{(j)}=\hat{p}_{\hat{K}}^\ast$ since there might be other cells that were not deemed safe by Problem~\ref{prb:MILP_truncated} $(z^{(j)}=0)$ but contains scenarios that might be safe $\left(1_{\mcR}\left(X(x_{0},U^{*}_{\hat{K}},W^{(i)})\right)=1\right)$. Hence, $\hat{p}\geq \hat{p}_{\hat{K}}^\ast$.
	\end{proof}

	%%%%%%%%%%%%%%%%%%%%%%%%%%%%%%%%%%%%%%%%%%%%%%%%%%%%% 
	%%%%%%%%%%%%%%%%%% Implementation %%%%%%%%%%%%%%%%%%%%%
	%%%%%%%%%%%%%%%%%%%%%%%%%%%%%%%%%%%%%%%%%%%%%%%%%%%%% 
	
	\subsection{Implementation} 
	\label{sec:Implementation}
	
	\begin{algorithm}%[h!]
		\begin{algorithmic}[]
			\STATE \bf{Input:}  {\textmd{LTI system~\eqref{eq:sys}, safe set $\mcS$, target set $\mcT$, initial state $x_0$.}}
			\STATE 
            \STATE \emph{Offline (independent of $x_0$):}
			\STATE {\begin{enumerate}
					\item {\textmd{Generate $\W_{K}$ by taking $K$ i.i.d. samples from $(\eta_{w})^{N}$.}}
					\item {\textmd{Construct $\Phi_{K}=\phi(\W_{K})$ with $\phi(W):=G_{w}W$.}}
\item {\textmd{Select $\hat{K}$ based on the required time complexity or from the $\mathrm{WSS}$ vs. $\hat{K}$ curve.}}
					\item {\textmd{Compute $\Psi_{\hat{K}}$, the optimal $\hat{K}$ seeds of $\Phi_{K}$, by a clustering method.}}
					\item {\textmd{Determine $\mcV_{\Phi_{K}}(\Psi_{\hat{K}})$ with cells $V_{\Phi_{K}}^{(1)}(\Psi_{\hat{K}}),\cdots,V_{\Phi_{K}}^{(\hat{K})}(\Psi_{\hat{K}})$ from \eqref{eq:V_j}.}}
					\item {\textmd{Compute importance rate vector $\alpha=\{\alpha^{(1)},\cdots,\alpha^{(\hat{K})}\}$ with $\alpha^{(j)}$ the number of elements of $V_{\Phi_{K}}^{(j)}(\Psi_{\hat{K}})$ for $j\in\N_{[1,\hat{K}]}$.}}
					\item {\textmd{Compute  $\varepsilon^{(j)}$ for $j\in\N_{[1,\hat{K}]}$ using Lemma~\ref{lem:buffer}.}}
			\end{enumerate}}
			
			\STATE
			\STATE \emph{Online (depends on $x_0$):}
			\STATE{\begin{enumerate}
					\item {\textmd{Solve Problem~\ref{prb:MILP_truncated} for $U^{*}_{\hat{K}}$.}}
					\item {\textmd{Compute $\hat{p}$ from \eqref{eq:phat}. }}
			\end{enumerate}}
			
			\STATE 
			\STATE \bf{Output:} $\hat{p}$.
		\end{algorithmic}
		\caption{Proposed Voronoi-based reach-avoid solution}
		\label{agm:Voronoi}
	\end{algorithm}
	
    Algorithm~\ref{agm:Voronoi} describes the proposed Voronoi-based method to solve the open-loop terminal time problem. Given a sample set $\W_K$ with $K$ random samples directly drawn from $(\eta_{w})^{N}$, one can construct $\phi(\W_K)$ and find its optimal $\hat{K}$ seeds. The $k$-means method can be used to find the seeds of a Voronoi partition as explained in Section~\ref{sec:VoronoiClustering}. In addition, in order to determine the number of required seeds, $\mathrm{WSS}$ can be used as a measure of variability of points in a cluster. A smaller $\mathrm{WSS}$ implies more compact clusters which reduces the size of defined buffers $\varepsilon^{(1)},\cdots,\varepsilon^{(\hat{K})}$ and the average number of samples in each cell. Note that by increasing $\hat{K}$, clusters become smaller and the precision of Problem~\ref{prb:MILP_truncated} grows. However, eventually,  the improvement precision is insignificant compared to the imposed computational complexity.  Therefore, we compute WSS as a function of $\hat{K}$, to explore this trade-off. We propose that the ``knee" of the curve provides an efficient compromise between precision and computational complexity. It is shown experimentally in the next section that the knee of $\mathrm{WSS}$ vs. $\hat{K}$ curve can be a good representative of the knee on the $\hat{p}$ vs. $\hat{K}$. As a result, $\hat{K}$ can be computed and selected in advance.  
	
    After selecting $\hat{K}$ based on the required running time for the real-time process or based on the $\mathrm{WSS}$ vs. $\hat{K}$, one can compute the Voronoi-based partition and the number of elements in each cell, and then compute the buffers from Lemma~\ref{lem:buffer}. All these steps are executed offline (independent of $ x_0$), while solving Problem~\ref{prb:MILP_truncated} and probability reconstruction using \eqref{eq:phat} is done online (dependent on $x_0$).   
	
	\section{Illustrative Example: Spacecraft Rendezvous} \label{sec:IllustrativeExample}
	
We consider the spacecraft rendezvous example discussed in \cite{lesser2013stochastic}. In this example, two spacecraft are in the same elliptical orbit. One spacecraft, referred to as the deputy, must approach and dock with another spacecraft, referred to as the chief, while remaining in a line-of-sight cone, in which accurate sensing of the other vehicle is possible.  The relative dynamics are described by the Clohessy-Wiltshire-Hill (CWH) equations as given in \cite{wiesel1989spaceflight}, 
\begin{align}
    \ddot{x} - 3 \omega x - 2 \omega \dot{y} = m_{d}^{-1}F_{x},\qquad\ddot{y} + 2 \omega \dot{x} = m_{d}^{-1}F_{y}.
  \label{eq:2d-cwh}
\end{align}
The position of the deputy is denoted by $x,y \in \mathbb{R}$ when the chief, with the mass $m_d=300$ kg, is located at the origin. For the gravitational constant $\mu$ and the orbital radius of the spacecraft $R_{0}$, $\omega = \sqrt{\mu/R_{0}^{3}}$ represents the orbital frequency. In this example, the spacecraft is in a circular orbit at an altitude of $850$ km above the earth.

We define $\zeta = [x,y,\dot{x},\dot{y}] \in \mathbb{R}^{4}$ as the system state and $u = [F_{x},F_{y}] \in \mathcal{U}\subseteq\mathbb{R}^{2}$ as the system input, then discretize the dynamics (\ref{eq:2d-cwh}) with a sampling time of $20$ s to obtain the discrete-time LTI system,
\begin{equation}
  \zeta_{t+1} = A \zeta_{t} + B u_{t} + w_{t}.
 \label{eq:lin-cwh}
\end{equation}
The additive stochastic noise, modeled by the Gaussian i.i.d. disturbance $w_{t} \in \mathbb{R}^{4}$, with $\mathbb{E}[w_{t}] = 0$, and $\mathbb{E}[w_{t}w_{t}^\top] = 10^{-4}\times\mbox{diag}(1, 1, 5 \times 10^{-4}, 5 \times 10^{-4})$, accounts for disturbances and model uncertainty.

We define the target set and the safe set as in \cite{lesser2013stochastic},
\begin{align}
  \mathcal{T} &= \left\{ \zeta \in \mathbb{R}^{4}: |\zeta_{1}| \leq 0.1, -0.1 \leq \zeta_{2} \leq 0, |\zeta_{3}| \leq 0.01, |\zeta_{4}| \leq 0.01 \right\}, \\
  \mathcal{S} &= \left\{\zeta \in \mathbb{R}^{4}: |\zeta_{1}| \leq \zeta_{2},
  -1\leq \zeta_{2}, |\zeta_{3}| \leq 0.05, |\zeta_{4}| \leq 0.05 \right\},
\end{align}
with a horizon of $N = 5$. We consider the initial position $x=y=-0.75$ km, the initial velocity $\dot{x}=\dot{y}=0$ km/s and $\mathcal{U} = [-0.1, 0.1]\times [-0.1, 0.1]$. The terminal time probability for this problem using existing approaches~\cite{VinodLCSS2017, lesser2013stochastic} is known to be $0.86$, which we assume to be the best open-loop controller-based reach-avoid probability estimate.

We set $K=2000$ as the number of original samples to estimate the terminal time probability,  with guarantees afforded by Theorem~\ref{thm:K_lb}, and run 100 random experiments in which in each experiment $\W_N$ is generated randomly. Simulations are carried out using CVX~\cite{cvx} on a $2.8$ GHz processor Intel Core i$5$ with $16$ GB RAM. Figure~\ref{fig:simulationresults}a shows the $\mathrm{WSS}$ curve (mean value and standard deviation of the results of the 100 experiments) with up to $100$ cells. Figure~\ref{fig:simulationresults}b shows the terminal time probability approximation provided by Algorithm~\ref{agm:Voronoi}. 
As proposed, the ``knee'' of Figure~\ref{fig:simulationresults}a coincides the ``knee'' of Figure~\ref{fig:simulationresults}b; improvements in the accuracy of Algorithm~\ref{agm:Voronoi} are insignificant beyond $\hat{K}=20$. In practice, 
$\hat{K}$ can be selected from one single experiment in which $WSS$ is calculated for a random disturbance set $\W_K$ for up to 100 (maximum allowable) cells. The computation of $\mathrm{WSS}$ curve shown in Figure~\ref{fig:simulationresults}a, for one experiment using $k$-means method, took only about $2.68$ s, hence is   reasonable for \textit{offline} computation to select $\hat{K}$. The reported time includes the computation time for solving $100$ $k$-means with $1$ to $100$ cells. This time, which is associated with offline step, can be further reduced by changing the step size of $\hat{K}$ variation (horizontal axis) or calculating $\mathrm{WSS}$ for arbitrary $\hat{K}$s (e.g. finer steps at the beginning and coarser steps at the end). The run time for the online component of Algorithm~\ref{agm:Voronoi} is shown in Figure~\ref{fig:simulationresults}c. Since Problem~\ref{prb:MILP_truncated} is a mixed-integer linear program,  the time complexity exponentially increases exponentially with the number of cells~\cite[Rem. 1]{bemporad_control_1999}.
   
\begin{figure}[t!]
	\vspace{0mm}
	\centering
	\begin{minipage}{3in}
		\centering
		(a)
		\includegraphics[width=3in]{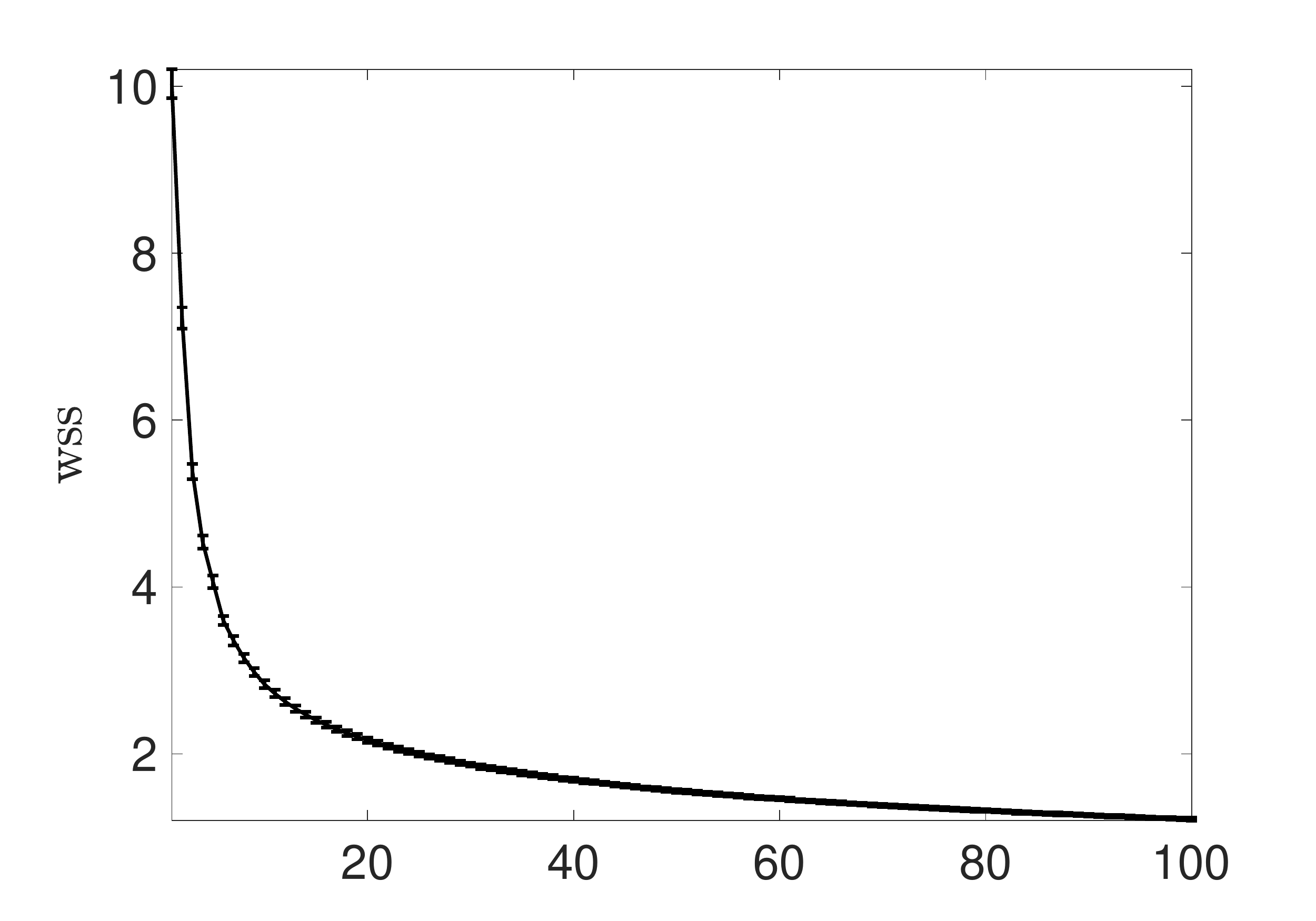}\\
	\end{minipage}
	\centering
	\begin{minipage}{3in}
		\centering 
        (b)
		\includegraphics[width=3in]{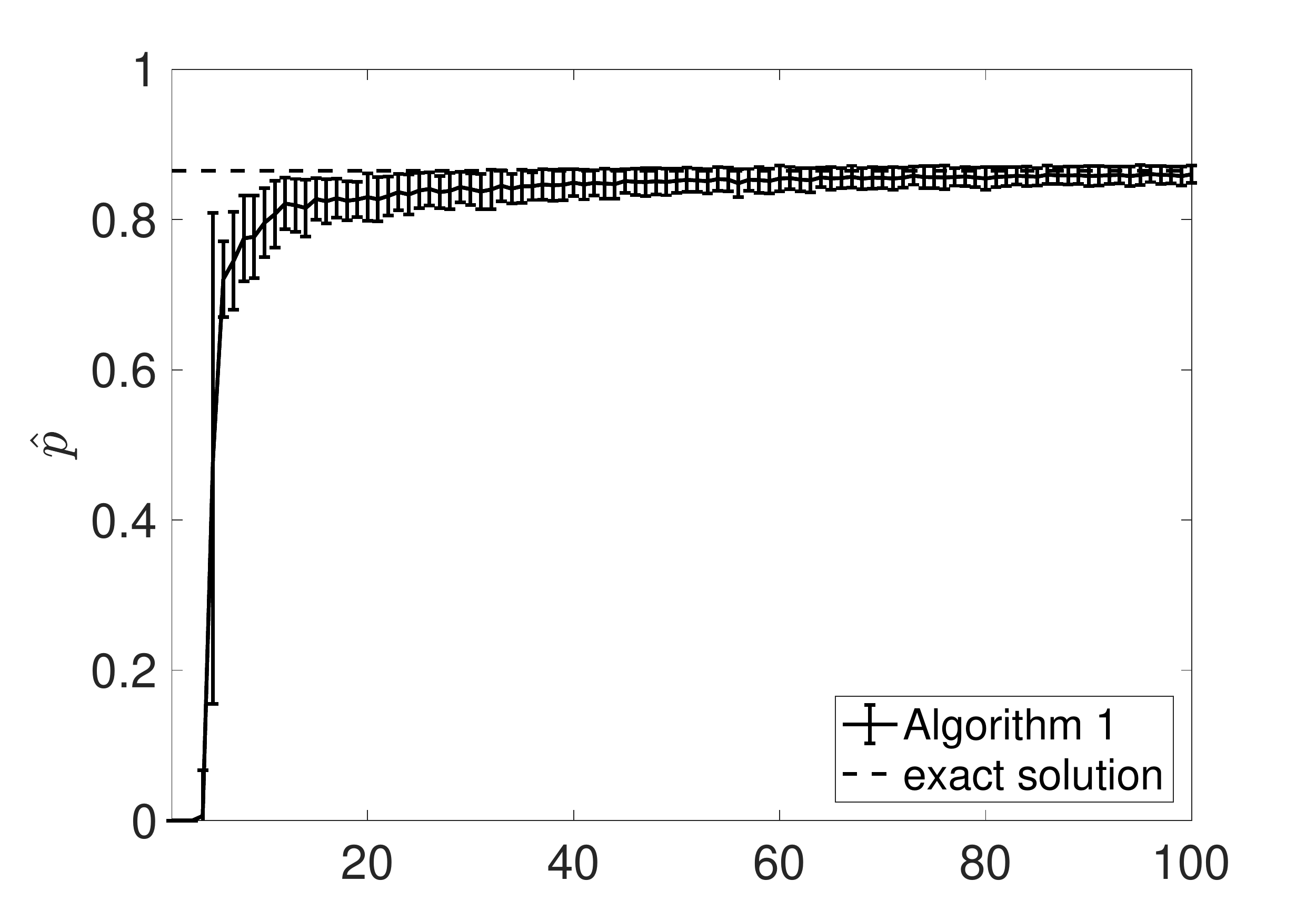}\\
	\end{minipage}
	\centering
	\begin{minipage}{3in}
		\centering
		(c)
		\includegraphics[width=3in]{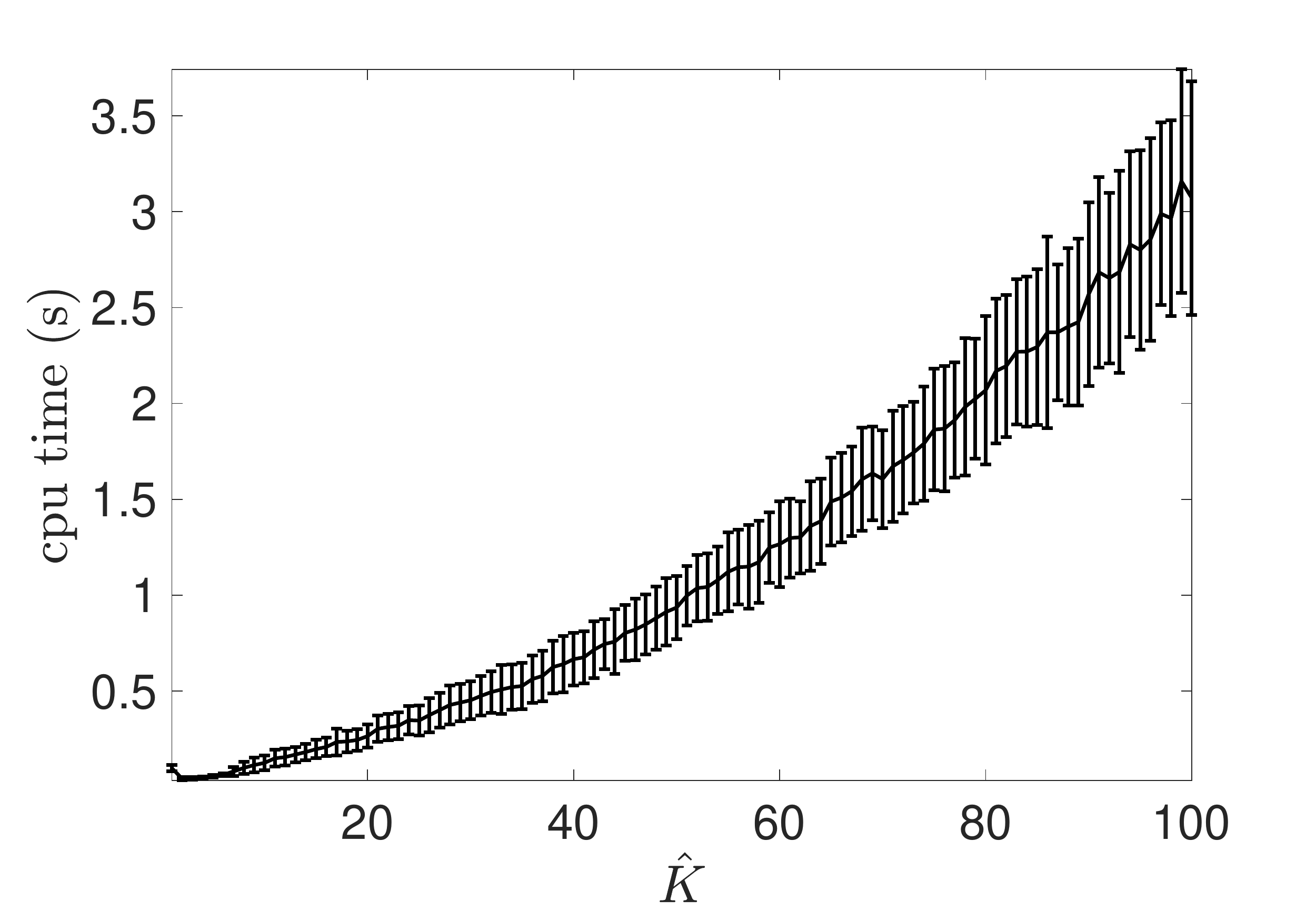}\\
	\end{minipage}
    \caption{ Mean and standard deviation of (a) within-cluster sum of squares (used to select $\hat{K}$, offline) (b) terminal time probability (c) online run time with increasing number of cells $\hat{K}$, obtained from $100$ experiments with $2000$ original scenarios.}
	\label{fig:simulationresults}
\end{figure}

We see in Figure~\ref{fig:simulationresults}b, that partitions with $20$ to $40$ cells provide a reasonable estimate of the terminal time probability, without significant loss of precision. The computed terminal time probability and the mean value of the online run time for $\hat{K}=20$, $40$ and $100$ are reported in Table~\ref{tab:compare}.
As desired, Algorithm~\ref{agm:Voronoi} provides a flexible trade-off between the accuracy and computation time by selecting a suitable partition, and can be significantly faster than the existing Fourier transform approach~\cite{VinodLCSS2017} and particle filter~\cite{lesser2013stochastic, blackmore2010probabilistic}   (which fails to deal with large $K$s due to the exponential complexity of MILP problem).

Figure~\ref{fig:trajectory_xy} shows the position trajectory, associated with $\zeta_1$ and $\zeta_2$, obtained by the Fourier method \cite{VinodLCSS2017} (blue dots) and the proposed Voronoi partition-based method (green stars) with $40$ cells. Green regions show the uncertainty regions of Voronoi method at different time instants obtained by 2000 original scenarios. 

\begin{table}
    \centering
    \begin{tabular}{|m{2.8cm}|m{2cm}|m{1.5cm}|}
    \hline									% For the horizontal line above. 
    Method & Terminal reach-avoid probability & Online run time (s) \\ \hline	% Use \cline{2-3} for partial horizontal line
    Algorithm~\ref{agm:Voronoi} & &\\
$K= 2000, \hat{K}=20$  &0.83 & 0.2 \\
$K=2000, \hat{K}=40$  &0.8492 & 0.6  \\  
$K=2000, \hat{K}=100$  &0.8604 & 2.7   \\ \hline
    Particle filter~\cite{lesser2013stochastic, blackmore2010probabilistic} (Problem~\ref{prb:MILP_SamplingBased}), $K=2000$& -  &-  \\ \hline
    Fourier transform~\cite{VinodLCSS2017} & 0.862 & 66 \\ \hline
   \end{tabular}
   \caption[caption here]{Terminal reach-avoid probability estimate and computation time of existing methods and Algorithm~\ref{agm:Voronoi}}
   \label{tab:compare}
\end{table}

	\begin{figure}
		\centering
		\includegraphics[width=3.0in]{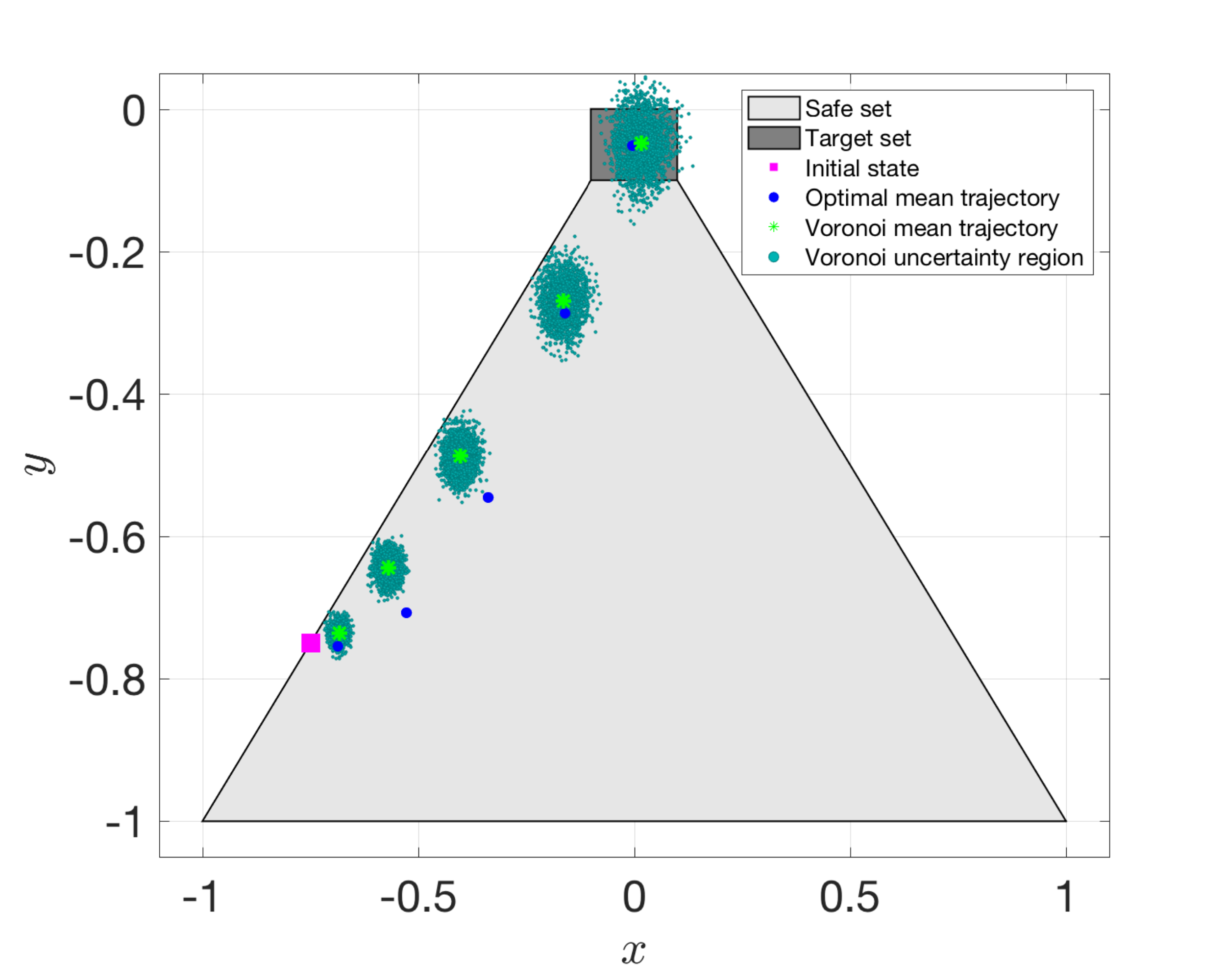}
\vspace{-2mm}
\caption{Position trajectory for Fourier algorithm given in \cite{VinodLCSS2017} and the proposed Voronoi partition-based method with 40 cells. }
		\label{fig:trajectory_xy}
	\end{figure}

	%%%%%%%%%%%%%%%%%%%%%%%%%%%%%%%%%%%%%%%%%%%%%%%%%%%%% 
	%%%%%%%%%%%%%%%%%%%%%%     Conclusion    %%%%%%%%%%%%%%%%%%%%
	%%%%%%%%%%%%%%%%%%%%%%%%%%%%%%%%%%%%%%%%%%%%%%%%%%%%% 

	\section{Conclusion} \label{sec:Conclusion}  
	In this paper we presented a novel partition-based method for under-approximating the terminal time probability through sample reduction. By using Hoeffding's inequality, we provided a bound on the required number of scenarios to achieve a desired probabilistic bound on the approximation error. Furthermore, we proposed a method which clusters the taken scenarios in few cells, each cell represented by a seed, where the number of cells is selected by the user in a systematic manner using the trend of a given curve or based on the desired running time. The proposed method scales easily with dimension since the clustering computational complexity increases linearly with the dimension of data. In addition, the simulation results confirm that the proposed method significantly decrease the running time, and therefore, it can be easily applied to real-time systems.  
	
	\bibliographystyle{IEEEtran}
	\bibliography{IEEEabrv,shortIEEE,StochasticReachAvoid}

\begin{thebibliography}{10}
\providecommand{\url}[1]{#1}
\csname url@samestyle\endcsname
\providecommand{\newblock}{\relax}
\providecommand{\bibinfo}[2]{#2}
\providecommand{\BIBentrySTDinterwordspacing}{\spaceskip=0pt\relax}
\providecommand{\BIBentryALTinterwordstretchfactor}{4}
\providecommand{\BIBentryALTinterwordspacing}{\spaceskip=\fontdimen2\font plus
\BIBentryALTinterwordstretchfactor\fontdimen3\font minus
  \fontdimen4\font\relax}
\providecommand{\BIBforeignlanguage}[2]{{%
\expandafter\ifx\csname l@#1\endcsname\relax
\typeout{** WARNING: IEEEtran.bst: No hyphenation pattern has been}%
\typeout{** loaded for the language `#1'. Using the pattern for}%
\typeout{** the default language instead.}%
\else
\language=\csname l@#1\endcsname
\fi
#2}}
\providecommand{\BIBdecl}{\relax}
\BIBdecl

\bibitem{SummersAutomatica2010}
S.~Summers and J.~Lygeros, ``Verification of discrete time stochastic hybrid
  systems: {A} stochastic reach-avoid decision problem,'' \emph{Automatica},
  vol.~46, no.~12, pp. 1951--1961, 2010.

\bibitem{HomChaudhuriACC2017}
B.~HomChaudhuri, A.~P. Vinod, and M.~Oishi, ``Computation of forward stochastic
  reach sets: Application to stochastic, dynamic obstacle avoidance,'' in
  \emph{American Control Conf.}, Seattle, WA, 2017.

\bibitem{MaloneHSCC2014}
N.~Malone, K.~Lesser, M.~Oishi, and L.~Tapia, ``Stochastic reachability based
  motion planning for multiple moving obstacle avoidance,'' in \emph{Proc.
  Hybrid Syst.: Comput. and Ctrl.}, 2014, pp. 51--60.

\bibitem{lesser2013stochastic}
K.~Lesser, M.~Oishi, and R.~S. Erwin, ``Stochastic reachability for control of
  spacecraft relative motion,'' in \emph{Proc. IEEE Conf. Dec. \& Ctrl.}\hskip
  1em plus 0.5em minus 0.4em\relax IEEE, 2013, pp. 4705--4712.

\bibitem{GleasonCDC2017}
\BIBentryALTinterwordspacing
J.~Gleason, A.~Vinod, and M.~Oishi, ``Underapproximation of reach-avoid sets
  for discrete-time stochastic systems via {L}agrangian methods,'' in
  \emph{IEEE Conf. Dec. Ctrl.}, 2017. [Online]. Available:
  \url{https://arxiv.org/abs/1704.03555.}
\BIBentrySTDinterwordspacing

\bibitem{AbateAutomatica2008}
A.~Abate, M.~Prandini, J.~Lygeros, and S.~Sastry, ``Probabilistic reachability
  and safety for controlled discrete time stochastic hybrid systems,''
  \emph{Automatica}, vol.~44, no.~11, pp. 2724--2734, 2008.

\bibitem{AbateHSCC2007}
A.~Abate, S.~Amin, M.~Prandini, J.~Lygeros, and S.~Sastry, ``Computational
  approaches to reachability analysis of stochastic hybrid systems,'' in
  \emph{Proc. Hybrid Syst.: Comput. and Ctrl.}, 2007, pp. 4--17.

\bibitem{KariotoglouSCL2016}
N.~Kariotoglou, K.~Margellos, and J.~Lygeros, ``On the computational complexity
  and generalization properties of multi-stage and stage-wise coupled scenario
  programs,'' \emph{Syst. and Ctrl. Lett.}, vol.~94, pp. 63--69, 2016.

\bibitem{ManganiniCYB2015}
G.~Manganini, M.~Pirotta, M.~Restelli, L.~Piroddi, and M.~Prandini, ``Policy
  search for the optimal control of {Markov} {Decision} {Processes}: A novel
  particle-based iterative scheme,'' \emph{{IEEE} Trans. Cybern.}, pp. 1--13,
  2015.

\bibitem{VinodLCSS2017}
A.~Vinod and M.~Oishi, ``Scalable {Underapproximation} for the {Stochastic}
  {Reach}-{Avoid} {Problem} for {High}-{Dimensional} {LTI} {Systems} {Using}
  {Fourier} {Transforms},'' \emph{IEEE Ctrl. Syst. Letters.}, vol.~1, no.~2,
  pp. 316--321, 2017.

\bibitem{VinodHSCC2018}
------, ``Scalable underapproximative verification of stochastic {LTI} systems
  using convexity and compactness,'' in \emph{Proc. Hybrid Syst.: Comput. and
  Ctrl.}, 2018, pp. 1--10.

\bibitem{SDP_kariotgolou}
D.~Drzajic, N.~Kariotoglou, M.~Kamgarpour, and J.~Lygeros, ``A semidefinite
  programming approach to control synthesis for stochastic reach-avoid
  problems,'' in \emph{Int'l Workshop on Applied Verification for Continuous
  and Hybrid Syst.}, 2016, pp. 134--143.

\bibitem{blackmore2010probabilistic}
L.~Blackmore, M.~Ono, A.~Bektassov, and B.~C. Williams, ``A probabilistic
  particle-control approximation of chance-constrained stochastic predictive
  control,'' \emph{IEEE Trans. Robot.}, vol.~26, no.~3, pp. 502--517, 2010.

\bibitem{satipizadeh2018Automatica}
H.~Sartipizadeh and T.~L. Vincent, ``A new robust mpc using an approximate
  convex hull,'' \emph{Automatica}, 2018.

\bibitem{sartipizadeh2018CCTOOL}
H.~Sartipizadeh and B.~A{\c{c}}{\i}kme{\c{s}}e, ``Approximate convex hull based
  sample truncation for scenario approach to chance constrained trajectory
  optimization,'' in \emph{Proc. American Ctrl. Conf.}, 2018, pp. 4700--4705.

\bibitem{calafiore2013stochastic}
G.~C. Calafiore and L.~Fagiano, ``Stochastic model predictive control of {LPV}
  systems via scenario optimization,'' \emph{Automatica}, vol.~49, no.~6, pp.
  1861--1866, 2013.

\bibitem{calafiore2006scenario}
G.~C. Calafiore and M.~C. Campi, ``The scenario approach to robust control
  design,'' \emph{{IEEE} Trans. Autom. Ctrl.}, vol.~51, no.~5, pp. 742--753,
  May 2006.

\bibitem{bemporad_control_1999}
A.~Bemporad and M.~Morari, ``Control of systems integrating logic, dynamics,
  and constraints,'' \emph{Automatica}, vol.~35, no.~3, pp. 407--427, 1999.

\bibitem{BoydConvex2004}
S.~Boyd and L.~Vandenberghe, \emph{Convex optimization}.\hskip 1em plus 0.5em
  minus 0.4em\relax Cambridge Univ. Press, 2004.

\bibitem{hartigan1975clustering}
J.~A. Hartigan, \emph{Clustering algorithms}.\hskip 1em plus 0.5em minus
  0.4em\relax Wiley, 1975.

\bibitem{Aloise2009}
\BIBentryALTinterwordspacing
D.~Aloise, A.~Deshpande, P.~Hansen, and P.~Popat, ``{NP}-hardness of euclidean
  sum-of-squares clustering,'' \emph{Machine Learning}, vol.~75, no.~2, pp.
  245--248, May 2009. [Online]. Available:
  \url{https://doi.org/10.1007/s10994-009-5103-0}
\BIBentrySTDinterwordspacing

\bibitem{hartigan1979algorithm}
J.~A. Hartigan and M.~A. Wong, ``Algorithm as 136: A k-means clustering
  algorithm,'' \emph{J. Royal Statistical Society. Series C (Applied
  Statistics)}, vol.~28, no.~1, pp. 100--108, 1979.

\bibitem{hoeffding_probability_1963}
W.~Hoeffding, ``Probability {Inequalities} for {Sums} of {Bounded} {Random}
  {Variables},'' \emph{J. Amer. Statistical Asso.}, vol.~58, no. 301, pp.
  13--30, 1963.

\bibitem{prandini2000probabilistic}
M.~Prandini, J.~Hu, J.~Lygeros, and S.~Sastry, ``A probabilistic approach to
  aircraft conflict detection,'' \emph{IEEE Trans. {Intelligent} Transportation
  {Syst.}}, vol.~1, no.~4, pp. 199--220, 2000.

\bibitem{wiesel1989spaceflight}
W.~E. Weisel, \emph{Spaceflight dynamics}.\hskip 1em plus 0.5em minus
  0.4em\relax New York, McGraw-Hill Book Co, 1989, vol.~2.

\bibitem{cvx}
M.~Grant and S.~Boyd, ``{CVX}: Matlab software for disciplined convex
  programming, version 2.1,'' \url{http://cvxr.com/cvx}, Mar. 2014.

\end{thebibliography}
\end{document}